\documentclass[11pt]{article}

\usepackage[reqno, namelimits, sumlimits]{amsmath}
\usepackage{amssymb, amsthm}
\newtheorem{theorem}{Theorem}

\newtheorem{proposition}[theorem]{Proposition}
\newtheorem{remark}[theorem]{Remark}
\newcommand{\R}{{\mathbb R}}
\newcommand{\eps}{\varepsilon}

\def\1#1{\hbox{$\eqref{lv1}_{#1}$}}
\begin{document}

\title{Liouville theorems, universal estimates \\
 and periodic solutions \\
 for cooperative parabolic Lotka-Volterra systems}

\author{Pavol Quittner\footnote{Supported in part by VEGA grant 1/0319/15.}
\\ \\
 \small Department of Applied Mathematics and Statistics, Comenius
University \\
 \small Mlynsk\'a dolina, 84248 Bratislava, Slovakia \\
 \tt quittner@fmph.uniba.sk
}

\date{}

\maketitle

\begin{abstract} 
We consider positive solutions of cooperative parabolic Lotka-Volterra 
systems with equal diffusion coefficients, in bounded and unbounded domains.
The systems are complemented by the Dirichlet or Neumann boundary conditions.
Under suitable assumptions on the coefficients of the reaction terms, 
these problems possess both global solutions and
solutions which blow up in finite time.
We show that any solution $(u,v)$ defined on the time interval $(0,T)$
satisfies a universal estimate of the form 
$$u(x,t)+v(x,t)\leq C(1+t^{-1}+(T-t)^{-1}),$$ 
where $C$ does not depend on $x,t,u,v,T$. 
In particular, this bound guarantees global existence and boundedness
for threshold solutions lying on the borderline between
blow-up and global existence.
Moreover, this bound yields optimal blow-up rate estimates 
for solutions which blow up in finite time.
Our estimates are based on new Liouville-type theorems for 
the corresponding scaling invariant parabolic system 
and require an optimal restriction on the space dimension $n$: $n\leq5$.
As an application we also prove the existence of time-periodic positive 
solutions if the coefficients are time-periodic.
Our approach can also be used for more general parabolic systems.
\end{abstract}

\section{Introduction} \label{intro}
We consider nonnegative solutions of the Lotka-Volterra system
\begin{equation} \label{lv1}
\left.\begin{aligned}
u_t-d_1\Delta u &= u(a_1-b_1u+c_1v), \\
v_t-d_2\Delta v &= v(a_2-b_2v+c_2u), 
\end{aligned}\quad\right\}
\qquad x\in \Omega, \ t\in(0,T),
\end{equation}
where
$\Omega$ is a (possibly unbounded) domain in $\R^n$
with a uniformly $C^2$ smooth boundary $\partial\Omega$,
$T\in(0,\infty]$,
$d_1,d_2$ are positive constants and
\begin{equation} \label{aibici}
a_i,b_i,c_i\in L^\infty(\Omega\times(0,\infty))\ \hbox{ for }\ i=1,2.
\end{equation}
Except for some marginal results in 
Theorem~\ref{thm-liouv2} and Remark~\ref{rem2}
we will always assume
\begin{equation} \label{bc-const}
b_1,b_2,c_1,c_2>0, \quad c_1c_2>b_1b_2.
\end{equation}
By $\nu$ we denote the outer unit normal on $\partial\Omega$
and by \1D or \1N we denote system \eqref{lv1} complemented
by the Dirichlet boundary conditions $u=v=0$ on $\partial\Omega\times(0,T)$
or the Neumann boundary conditions $u_\nu=v_\nu=0$ on $\partial\Omega\times(0,T)$,
respectively.
Notice that \1D=\1N=\eqref{lv1} if $\Omega=\R^n$.
If $\Omega$ is bounded then by $\Lambda_1$ we denote
the least eigenvalue of the negative Dirichlet Laplacian in $\Omega$.
Except for Proposition~\ref{propBH}, 
by a solution we will always mean a nonnegative classical solution.

First consider the Dirichlet problem \1D and 
assume that $\Omega$ is bounded and 
the coefficients $a_i,b_i,c_i$ are constant.
Then some solutions of \1D blow up in finite time,
see \cite[the proof of Theorem 12.6.1]{Pao}. 
Blow-up rates of such solutions have been studied
in \cite{Lin,LW} but an upper estimate of the rate
(which is usually much more difficult
than a lower estimate) has only been established if $n=1$.
Under suitable additional assumptions, \1D possesses also
nontrivial global solutions and steady states:
If we assume
$a_1/d_1,a_2/d_2<\Lambda_1$, for example,
then the existence of global solutions follows from
the stability of the zero solution.
If, in addition, $n\leq5$ then there exists a positive steady state,
and the assumption $n\leq5$ is also necessary
if $\Omega$ is starshaped and $a_1/d_1=a_2/d_2$, see \cite{Lou}.

If one considers the Neumann problem \1N
with  $\Omega$ bounded and $a_i,b_i,c_i$ constant
then some solutions blow up again
and even ``diffusion-induced blow-up'' occurs:
There exist blow-up solutions of \1N 
such that the solutions of the corresponding system
of ODEs exist globally, see \cite{LNN}.
On the other hand, nontrivial global solutions also exist
if $a_1,a_2<0$, for example.

The existence of blow-up and global solutions of system \eqref{lv1}
is also known in the case of non-constant coefficients, see \cite{LLP}
and the references therein, for example.
If the problem \1D or \1N
possesses both global solutions
and solutions which blow up in finite time
then one can study so-called threshold solutions,
i.e.~solutions lying on the borderline between global existence
and blow-up.
The study of such solutions is difficult even for the scalar problem
\begin{equation} \label{Fuj1}
\left.\begin{aligned}
u_t-\Delta u &= cu^2, &\qquad& x\in\Omega, \ t>0,\\
 u &= 0, &\qquad& x\in\partial\Omega, \ t>0,
\end{aligned}\quad\right\}
\end{equation}
where $c$ is a positive constant and $\Omega$ is bounded.
Notice that problem \eqref{Fuj1} is just a very special case of \1D
(if we set $u=v$, $d_1=d_2=1$, $a_1=a_2=0$, $b_1=b_2=1$ and $c_1=c_2=c+1$).
All threshold solutions of \eqref{Fuj1} 
are global and bounded if $n\leq5$ 
but they may be global and unbounded
if $n=6$ and they even may blow up in finite time if $n>6$, 
see \cite{QS} and the references therein. 
It should be mentioned that the threshold solutions 
of \eqref{Fuj1} are --- in some sense --- the most interesting ones.
In particular, any positive steady state of \eqref{Fuj1}
is a threshold solution and 
the $\omega$-limit set of any global bounded positive threshold solution
of \eqref{Fuj1}
is nonempty and consists of positive steady states.
As far as we know, the behavior of threshold solutions
of problems \1D or \1N has not been studied yet.

If $n\leq5$ and the coefficient $c$ in \eqref{Fuj1}
is a positive periodic function of $t$ then one can use 
a priori estimates of global solutions of \eqref{Fuj1}
in order to prove the existence of nontrivial periodic solutions,
see \cite{Est,QNoDEA,BPQ}.
Again, the existence of periodic solutions does not seem to be known 
for problems \1D or \1N with periodic coefficients in the presence of blow-up
(see \cite{LLP}, for example, for the existence of periodic solutions
of \1D in the case $c_1c_2<b_1b_2$ which excludes blow-up).

In this paper we will assume that $n\leq5$, $d_1=d_2=1$,
the coefficients $a_i,b_i,c_i$ satisfy \eqref{bc-const}
and suitable regularity assumptions and we will prove that
all solutions of \1D or \1N satisfy 
universal a priori estimates of the form 
\begin{equation} \label{est1}
u(x,t)+v(x,t)\leq C(1+t^{-1}+(T-t)^{-1}),\qquad x\in\Omega,\ t\in(0,T),
\end{equation}
where $C$ does not depend on $x,t,u,v,T$.
These estimates guarantee global existence and boundedness
of threshold solutions and also optimal blow-up rate estimates
of solutions which blow up in finite time. 
In addition,
if $\Omega$ is bounded, $a_i,b_i,c_i$ are time-periodic and
$a_1,a_2<\Lambda_1$ then 
these estimates
combined with a homotopy argument
guarantee the existence of
a time-periodic positive solution of \1D.
The homotopy used in our proof 
is quite different
from those used for the scalar problem \eqref{Fuj1} in \cite{Est}
or the steady-state problem for \1D
in \cite{Lou}.

Our estimates are based on Liouville-type theorems
for corresponding scaling invariant systems.
In fact, we will prove Liouville-type theorems
for more general systems of the form \eqref{sys-MSS}.
Their proofs rely
on the fact that the components of any entire solution
of such system have to be proportional, i.e.
the problem can be reduced to a scalar problem.  
Arguments of this type have recently been used 
in \cite{MSS,QS-S} in the case of elliptic systems.

\section{Main results} \label{results}

We will first specify a class of coefficients of system \eqref{lv1}
such that the constant $C$ in the universal estimate \eqref{est1} 
will not depend on the coefficients in this class. 
Fix
$\eps_0,M_0>0$ and a continuous function
$\omega_0:[0,\infty)\to[0,\infty)$ satisfying $\omega_0(0)=0$,
and set
$$ \begin{aligned}
{\cal F} ={{\cal F}(\eps_0,M_0,\omega_0)}
 := \{\phi:&\Omega\times(0,\infty)\to[\eps_0,M_0]: \\
 & |\phi(x,t)-\phi(y,s)|\leq\omega_0(|x-y|+|t-s|)\\
 &\ \hbox{for all }(x,t),(y,s)\in\Omega\times(0,\infty)\}.
\end{aligned}$$ 
We assume that
\begin{equation} \label{bc1}
\left.
\begin{aligned}
 & a_1,a_2\in L^\infty(\Omega\times(0,\infty))\ \hbox{ satisfy }\ |a_1|,|a_2|\leq M_0, \\  
 & b_1,b_2,c_1,c_2\in {\cal F}(\eps_0,M_0,\omega_0)\ \hbox{ satisfy }\ c_1c_2\geq b_1b_2+\eps_0,
\end{aligned}
\quad\right\}
\end{equation}
and consider system \eqref{lv1} with $d_1=d_2=1$, i.e.  
\begin{equation} \label{lv11}
\left.\begin{aligned}
u_t-\Delta u &= u(a_1-b_1u+c_1v), \\
v_t-\Delta v &= v(a_2-b_2v+c_2u), 
\end{aligned}\quad\right\}
\qquad x\in \Omega, \ t\in(0,T).
\end{equation}
Our first result guarantees universal estimates of positive solutions of
\eqref{lv11}.

\begin{theorem} \label{thm-ub}
Let $\Omega$ be an arbitrary nonempty open subset of $\R^n$, 
$n\leq5$, $T\in(0,\infty]$,
$\eps_0,M_0>0$, and $\omega_0:[0,\infty)\to[0,\infty)$ be
a continuous function satisfying $\omega_0(0)=0$.
Assume also that $a_1,a_2,b_1,b_2,c_1,c_2$ satisfy \eqref{bc1}.
Then there exists a positive constant $C=C(\eps_0,M_0,\omega_0)$
such that any positive solution $(u,v)$ of \eqref{lv11} satisfies
\begin{equation} \label{est4}
 u(x,t)+v(x,t)\leq C\bigl(C_1+t^{-1}+(T-t)^{-1}
  +C_2\hbox{\rm dist}^{-2}(x,\partial\Omega)\bigr)
\end{equation}
for all $(x,t)\in \Omega\times(0,T)$,
where $C_1=C_2=1$,
the constant
$C$ is independent of  $x,t,u,v,\Omega,T,a_1,a_2,b_1,b_2,c_1,c_2$,
$$ (T-t)^{-1}:=0\ \hbox{ if }\ T=\infty,\qquad
\hbox{\rm dist}^{-2}(x,\partial\Omega):=0\ \hbox{ if }\ \Omega=\R^n,$$
and the solution $(u,v)$ need not satisfy any boundary or initial condition.

If $\Omega$ is uniformly $C^2$ smooth
and the solution $(u,v)$ satisfies 
the homogeneous Dirichlet or Neumann boundary condition 
on $\partial\Omega\times(0,T)$
then \eqref{est4} is true with $C=C(\Omega,\eps_0,M_0,\omega_0)$, 
$C_1=1$ and $C_2=0$, i.e.~estimate \eqref{est1} is true. 

If $a_1=a_2=0$  and $b_1,b_2,c_1,c_2$ are constants,
then \eqref{est4} is true with $C_1=0$ and $C_2=1$.
\end{theorem}

In particular,  Theorem~\ref{thm-ub} 
yields optimal upper blow-up rate estimates
of solutions which blow up at time $T$,
and also guarantees global existence and boundedness of
threshold solutions.

In order to prove the existence of periodic solutions
we fix $T\in(0,\infty)$ and
assume that
\begin{equation} \label{ass-per}
\left.\begin{aligned}
& \hbox{$\Omega\subset\R^n$ is a $C^3$-smooth bounded domain, $n\leq5$,} \\
& a_i,b_i,c_i\in C(\overline\Omega\times[0,\infty)) 
     \hbox{ are $T$-periodic in $t$ for } i=1,2, \\
& b_1,b_2,c_1,c_2>0, \quad c_1c_2>b_1b_2,
\quad a_1,a_2<\Lambda_1.  
\end{aligned}\quad\right\}
\end{equation}

\begin{theorem} \label{thm-per}
Assume \eqref{ass-per}.
Then system \eqref{lv11} complemented by the homogeneous Dirichlet boundary
conditions
possesses a positive $T$-periodic solution.
\end{theorem}

Similarly as in \cite{PQS2},
our estimates are based on doubling arguments, scaling
and Liouville-type theorems for corresponding 
scaling invariant problems.
In the case of system \eqref{lv11}
the scaling invariant problem is 
\begin{equation} \label{lv2}
\left.\begin{aligned}
u_t-\Delta u &= u(-b_1u+c_1v) \\
v_t-\Delta v &= v(-b_2v+c_2u) 
\end{aligned}\quad\right\}\qquad\hbox{in }X\times\R,
\end{equation}
where 
\begin{equation} \label{X}
\left.\begin{aligned}
\hbox{either }\  X&=\R^n \\
\hbox{or }\ X&=\R^n_+:=\{x=(x_1,x_2,\dots,x_n):x_1>0\},
\end{aligned}\quad\right\}
\end{equation}
$b_1,b_2,c_1,c_2$ are constants satisfying \eqref{bc-const},  
and  system \eqref{lv2} is complemented by the homogeneouos
Dirichlet or Neumann boundary conditions if $X=\R^n_+$.
In fact, we will 
assume \eqref{X} and consider more general systems of the form
\begin{equation} \label{sys-MSS}
\left.\begin{aligned}
u_t-\Delta u &= u^r(-b_1u^q+c_1v^q) \\
v_t-\Delta v &= v^r(-b_2v^q+c_2u^q), 
\end{aligned}\quad\right\}
\qquad \hbox{in }\ X\times\R,
\end{equation} 
complemented by the Dirichlet or Neumann boundary conditions if $X=\R^n_+$.
Here $b_1,b_2,c_1,c_2$ are constants satisfying \eqref{bc-const},
$q\geq r>0$, $q+r>1$ and
\begin{equation} \label{qr}
n\leq2\quad\hbox{or}\quad q+r<\frac{n(n+2)}{(n-1)^2}.
\end{equation}

\begin{theorem} \label{thm-liouv1}
Assume \eqref{X}.
Let $b_1,b_2,c_1,c_2$ be constants satisfying  \eqref{bc-const}
and $q,r$ satisfy $q\geq r>0$, $q+r>1$ and \eqref{qr}. 
Let $(u,v)$ be a nonnegative solution of \eqref{sys-MSS}
complemented by the homogeneouos
Dirichlet or Neumann boundary conditions if $X=\R^n_+$.
In the case of the Dirichlet boundary condition 
assume also that $(u,v)$ is bounded.
Then $u\equiv v\equiv0$.
\end{theorem}

In the proof of Theorem~\ref{thm-liouv1} 
we first show that there exists $K>0$ such that any nonnegative 
(or nonnegative bounded) solution satisfies $u=Kv$. 
This implies that $u$ is a solution of the scalar equation
\begin{equation} \label{FujX}
  u_t-\Delta u=cu^{q+r} \quad\hbox{in }X\times\R 
\end{equation}
with some $c>0$
(and $u$ satisfies the corresponding boundary condition if $X=\R^n_+$).
Now the Liouville theorems in \cite{BV,PQS2,Qpr} guarantee $u\equiv0$
if \eqref{qr} is true 
(or $q+r<(n+2)/(n-2)$ if we consider radial solutions only).
Let us also mention that
\eqref{FujX} possesses positive radial solutions
if $X=\R^n$ and $q+r\geq(n+2)/(n-2)$.

If $r=q=1$ (which corresponds to the 
Lotka-Volterra system \eqref{lv2})
then  \eqref{qr} can be written in the form
$n\leq5$ and condition $q+r\geq(n+2)/(n-2)$
is equivalent to $n\geq6$, hence
our Liouville theorems are optimal in this case.

Similarly, if $r=1$, $q=2$ 
then we obtain the nonexistence for \eqref{sys-MSS}
under the optimal condition $n\leq3$
and --- in the same way as in the case of $r=q=1$ ---
one can use this result to obtain
universal estimates for solutions of systems of the form
 \begin{equation} \label{Schr1}
\left.\begin{aligned}
u_t-\Delta u &= a_1u-b_1u^3+c_1uv^2, \\
v_t-\Delta v &= a_2v-b_2v^3+c_2u^2v, 
\end{aligned}\quad\right\}
\qquad x\in \Omega, \ t\in(0,T),
\end{equation}
complemented by the Dirichlet or Neumann boundary conditions. 
Notice that steady states of \eqref{Schr1}
correspond to the standing wave solutions of 
related Schr\"odinger systems,
and estimates of global solutions of \eqref{Schr1}
can be useful in the study of the steady states, see \cite{WW}.
Let us also mention that
the nonexistence of positive solutions of \eqref{sys-MSS} with $r=1$, $q=2$,
$b_1=b_2=-1$ and $c_1=c_2>-1$ has recently
been proved in \cite{Qpr}
for $n\leq2$ (or $n\leq3$ in the class of
radially symmetric solutions): The proof heavily used the gradient
structure of the system.
That result indicates that condition \eqref{bc-const}
is not necessary for the validity of Liouville theorems.
Therefore
in Section~\ref{sec-lt} we also briefly discuss
Liouville-type results 
for the Lotka-Volterra system \eqref{lv1} with coefficients
$b_1,b_2$ not necessarily positive and $d_1\ne d_2$,
see Theorem~\ref{thm-liouv2} and Remark~\ref{rem2}.

\section{Liouville theorems} \label{sec-lt}

We will use the following modification of \cite[Lemma 2.1]{Fold}.

\begin{proposition} \label{propF}
Assume \eqref{X}. 
Let $h\in C([0,\infty))$, $h(s)>0$ for $s>0$.
Let $w\in C^{2,1}(X\times\R)\cap C(\overline X\times\R)$ be bounded
and satisfy the inequality 
\begin{equation} \label{eq-w}
 (w_t-\Delta w)\,\hbox{\rm sign}(w)\leq -h(|w|) \qquad\hbox{in }X\times\R,
\end{equation}
and the boundary condition $w=0$ on $\partial\R^n_+\times\R$ if $X=\R^n_+$.  
Then $w\equiv0$.
\end{proposition}

\begin{proof}
Assume on the contrary $w\not\equiv0$. Since $-w$ also satisfies the assumptions
of Proposition~\ref{propF}, we may assume 
$$W:=\sup_{X\times\R}w >0.$$
Fix $(x^*,t^*)\in X\times\R$ such that $w(x^*,t^*)\geq W/2$.
For each $\eps>0$ set
$$ w_\eps(x,t):=w(x,t)-\eps|x-x^*|^2
  -\eps\bigl(\sqrt{(t-t^*)^2+1}-1\bigr),\qquad x\in X,\ t\in\R.$$
Since $w_\eps(x^*,t^*)=w(x^*,t^*)>0$,
$w_\eps(x,t)\to-\infty$ as $|x|+|t|\to\infty$,
$w_\eps(x,t)<0$ if $X=\R^n_+$ and $x\in\partial\R^n_+$,
there exists 
$(x_\eps,t_\eps)\in X\times\R$ satisfying 
$w_\eps(x_\eps,t_\eps)=\sup_{X\times\R}w_\eps$.
Notice that
$$ W\geq w(x_\eps,t_\eps)\geq w_\eps(x_\eps,t_\eps)
  \geq w_\eps(x^*,t^*)=w(x^*,t^*)\geq \frac{W}2>0, $$
and
$$ (w_\eps)_t(x_\eps,t_\eps)=0, \qquad \Delta w_\eps(x_\eps,t_\eps)\leq0.$$
Consequently,
$$ \begin{aligned}
  0 &\leq (w_\eps)_t(x_\eps,t_\eps)-\Delta w_\eps(x_\eps,t_\eps) \\
    &= w_t(x_\eps,t_\eps)-\Delta w(x_\eps,t_\eps)
       -\eps\frac{t_\eps-t^*}{\sqrt{(t_\eps-t^*)^2+1}}+2\eps n \\
    &\leq -h(w(x_\eps,t_\eps))+\eps+2\eps n \\
    &\leq -\inf_{W\geq s\geq W/2} h(s)+\eps+2\eps n.
\end{aligned} $$
Since the first term on the right hand side is negative and independent of $\eps$,
we arrive at a contradiction if $\eps$ is small enough.
\end{proof}

Now we are ready to prove Liouville-type theorems for system \eqref{sys-MSS}.

\begin{proof}[Proof of Theorem~\ref{thm-liouv1}]
If $X=\R^n_+$ and
$(u,v)$ is a nonnegative solution of the Neumann problem
then extending $(u,v)$ by 
$$\left.\begin{aligned}
u((-x_1,x_2,\dots,x_n),t)&:=u((x_1,x_2,\dots,x_n),t), \\
v((-x_1,x_2,\dots,x_n),t)&:=v((x_1,x_2,\dots,x_n),t),
\end{aligned}\ \right\}
\quad x\in\R^n_+,\ t\in\R,
$$
we obtain a nonnegative solution of \eqref{sys-MSS} with $X=\R^n$.
Consequently, it is sufficient to consider the Dirichlet problem
and the case $X=\R^n$.

If $X=\R^n$
and $(u,v)$ is a nonnegative solution of \eqref{sys-MSS} 
which is not identically zero then by doubling and scaling
arguments we may assume that $(u,v)$ is bounded. 
In fact, assume that $(u+v)(x_k,t_k)\to\infty$
for some $(x_k,t_k)\in\R^n\times\R$.
Set  $M:=(u+v)^{(q+r-1)/2}$.
Then the Doubling lemma \cite[Lemma 5.1]{PQS1}
guarantees the existence of $(\tilde x_k,\tilde t_k)$
such that
$M(\tilde x_k,\tilde t_k)\geq M(x_k,t_k)\to \infty$
and 
$M(x,t)\leq 2M(\tilde x_k,\tilde t_k)$
for all $(x,t)$ satisfying 
$|x-\tilde x_k|+(t-\tilde t_k)^{1/2}\leq k\lambda_k$,
where $\lambda_k:=1/M(\tilde x_k,\tilde t_k)$.
It is easily seen that the rescaled functions
$$ \begin{aligned}
\tilde u(y,s)&:=\lambda_k^{2/(q+r-1)}u(\tilde x_k+\lambda_ky,\tilde t_k+\lambda_k^2s),\\
\tilde v(y,s)&:=\lambda_k^{2/(q+r-1)}v(\tilde x_k+\lambda_ky,\tilde t_k+\lambda_k^2s)
\end{aligned}$$
converge locally uniformly to a nonnegative nontrivial bounded solution
of \eqref{sys-MSS}.

Hence, we may assume that
$(u,v)$ is a nonnegative bounded solution of \eqref{sys-MSS}
with $X=\R^n$ or a nonnegative bounded solution of
the Dirichlet problem.
Now \cite[Lemma 7.1(i)]{MSS}
guarantees the existence of $K,C>0$ such that
the function $w:=u-Kv$ satisfies
$$(w_t-\Delta w)\,\hbox{sign}(w)\leq-C(u+Kv)^{q+r-1}|w|\leq-C|w|^{q+r},$$
hence Proposition~\ref{propF} yields $u=Kv$.
Our assumption $c_1c_2>b_1b_2$ guarantees that $u$ 
solves the scalar equation
$$ u_t-\Delta u = cu^{r+q} $$
with some $c>0$
(and satisfies the Dirichlet boundary condition if $X=\R^n_+$).
Consequently, it is sufficient to use
the Liouville theorems in \cite{BV,PQS2,Qpr}.
\end{proof}

\begin{remark} \label{remKC} \rm
Assume $r=1$.
Then the constant $K$ in the proof of Theorem~\ref{thm-liouv1}
can be computed explicitly: $K=[(c_1+b_2)/(c_2+b_1)]^{1/q}$
(see \cite{MSS}). 
Notice also that 
if $r=q=1$ and $w=u-Kv$ then $w_t-\Delta w=-(b_1u+b_2v)w$.
\end{remark}

In the proof of the existence of periodic solutions we will also need
estimates based on
the following Liouville theorem.

\begin{theorem} \label{thm-liouv-per}
Assume \eqref{X} and $n\leq5$.
Let $b_1,b_2,c_1,c_2$ be constants satisfying  \eqref{bc-const},
$\lambda\in[0,1]$ and $K:=(c_1+b_2)/(c_2+b_1)$.
Let $(u,v)$ be a nonnegative solution of the system
\begin{equation} \label{lv3}
\left.\begin{aligned}
u_t-\Delta u &= \lambda u(-b_1u+c_1v)+(1-\lambda)K^3v^2 \\
v_t-\Delta v &= \lambda v(-b_2v+c_2u)+(1-\lambda)u^2 
\end{aligned}\quad\right\}\qquad\hbox{in }X\times\R,
\end{equation}
complemented by the homogeneous
Dirichlet or Neumann boundary conditions if $X=\R^n_+$.
In the case of the Dirichlet boundary condition 
assume also that $(u,v)$ is bounded.
Then $u\equiv v\equiv0$.
\end{theorem} 

\begin{proof}
The proof is almost the same as the proof of Theorem~\ref{thm-liouv1}
with $q=r=1$. Due to Remark~\ref{remKC}, the function $w:=u-Kv$ satisfies
$$ \begin{aligned}
  w_t-\Delta w &=-\lambda(b_1u+b_2v)w+(1-\lambda)K(K^2v^2-u^2) \\
               &= -\bigl(\lambda(b_1u+b_2v)+(1-\lambda)K(u+Kv)\bigr)w
\end{aligned} $$
hence 
$$(w_t-\Delta w)\,\hbox{sign}(w)\leq-\tilde C|w|^2,$$
and $u=Kv$ due to Proposition~\ref{propF}.
Consequently, $u$ solves the scalar equation
$$ u_t-\Delta u = (\lambda c+(1-\lambda)K)u^2, $$
where $c=c_1/K-b_1>0$.
The scalar Liouville theorems in \cite{BV,PQS2} guarantee $u\equiv0$. 
\end{proof}

In the rest of this section we consider scaling invariant problems
corresponding to the Lotka-Volterra systems 
without assumption \eqref{bc-const} 
or in the case of unequal diffusion coefficients $d_1\ne d_2$.

\begin{theorem} \label{thm-liouv2}
Assume $b_1=b_2=0$, $c_1,c_2>0$ and $n\leq5$.
Let $(u,v)$ be a nonnegative bounded solution of \eqref{lv2} with $X=\R^n$.
Then either $(u,v)\equiv (C,0)$  or $(u,v)\equiv(0,C)$ 
for some $C\geq0$.
\end{theorem}

\begin{proof}
Scaling arguments show that we may assume $c_1=c_2=1$.
The function $w:=u-v$ is a bounded entire solution
of the linear heat equation hence $w\equiv D$ for some constant $D$
(see \cite[Theorem 1]{Eid}).
W.l.o.g.\ we may assume $D\geq0$.
If $D=0$ then $u=v$ and the Liouville theorem 
\cite[Theorem 21.2]{QS} guarantees $u\equiv v\equiv0$.
If $D>0$ then given $p\in(1,\min(2,1+2/n)]$ there exists $d=d(p,D)>0$ such that  
$$ v_t-\Delta v=uv=(v+D)v\geq dv^p, $$
and the Fujita theorem \cite[Theorem 18.1]{QS}
together with the comparison principle
imply $v\equiv0$. 
\end{proof}

The existence of semi-trivial entire solutions 
of the form $(C,0)$ and $(0,C)$ with $C>0$
disables one 
to use standard scaling arguments to
prove a priori estimates of solutions in a straightforward way.
However, at least in the case of similar elliptic systems,
existence of semi-trivial entire solutions 
represents just a technical difficulty
and the scaling arguments do apply,   
see \cite{Zou}.

\begin{remark} \label{rem2} \rm
If one considers \eqref{lv2} with $X=\R^n$, $b_1,b_2\leq0$ and $c_1,c_2>0$,
for example, then the function
$w:=\sqrt{uv}$ satisfies $w_t-\Delta w\geq cw^2$ for some $c>0$,
hence the Fujita theorem
\cite[Theorem 18.1]{QS} guarantees $w\equiv0$ if $n\leq 2$
(and similar result can be obtained for the Dirichlet
problem in the halfpace if $n=1$,
see \cite[Remark 18.6(i)]{QS}).
If $b_1,b_2\ne0$ then this implies
$u\equiv v\equiv0$, and these Liouville theorems for \eqref{lv2}
enable one to prove universal estimates
of solutions of \eqref{lv1} with $d_1=d_2=1$.
In addition, the Fujita-type theorems mentioned above
and comparison with suitable subsolutions
enables one to prove the Liouville theorems
even for the generalization of \eqref{lv2}
with unequal diffusion coefficients $d_1\ne d_2$,
see \cite{Lin}.
However, the restrictions $n\leq2$ and $n=1$ (in the case
of the Dirichlet problem) seem to be far from optimal.
\end{remark}
 
\section{Universal estimates} \label{sec-ue}

\begin{proof}[Proof of Theorem~\ref{thm-ub}]
The proof follows those of \cite[Theorem 3.1 and 4.1]{PQS2}
and we just sketch it.

In order to prove estimate \eqref{est4} with $C_1=C_2=1$
we will follow the proof of \cite[Theorem 3.1(i)]{PQS2}.
Assume that estimate \eqref{est4} fails.
Then, for $k=1,2,\dots$, 
there exist nonempty open sets $\Omega_k$, $T_k\in(0,\infty]$,
coefficients $a_{i,k},b_{i,k},c_{i,k}$, $i=1,2$,
satisfying \eqref{bc1} with $\Omega$ replaced by $\Omega_k$,
solutions $(u_k,v_k)$ of \eqref{lv11} with $\Omega,T,a_1,a_2,b_1,b_2,c_1,c_2$
replaced by $\Omega_k,T_k,a_{1,k},a_{2,k},b_{1,k},b_{2,k},c_{1,k},c_{2,k}$
and points $(y_k,\tau_k)\in D_k:=\Omega_k\times(0,T_k)$ such that
\begin{equation} \label{Mkbig}
M_k(y_k,\tau_k)>2k(1+d_P^{-1}((y_k,\tau_k),\partial D_k)),
\end{equation}
where  $M_k:=\sqrt{u_k+v_k}$ and 
$$d_P((x,t),(y,\tau)):=|x-y|+|t-s|^{1/2}$$
denotes the parabolic distance. The Doubling lemma 
\cite[Lemma 5.1]{PQS1} guarantees the existence of $(x_k,t_k)\in D_k$
such that
$$
\begin{aligned}
 M_k(x_k,t_k) &\geq M_k(y_k,\tau_k),\quad
  M_k(x_k,t_k)>2kd_P^{-1}((x_k,t_k),\partial D_k), \\
 M_k(x,t) &\leq M_k(x_k,t_k) 
 \quad\hbox{whenever}\quad d_P((x,t),(x_k,t_k))\leq k\lambda_k,
\end{aligned}
$$
where
$$ \lambda_k:=M_k^{-1}(x_k,t_k)\to 0.$$
Set
$$ \begin{aligned}
   \tilde u_k(y,s) &:=\lambda_k^2u_k(x_k+\lambda_ky,t_k+\lambda_k^2s), \\
   \tilde v_k(y,s) &:=\lambda_k^2v_k(x_k+\lambda_ky,t_k+\lambda_k^2s), \\
   \tilde a_{1,k}(y,s) &:= a_{1,k}(x_k+\lambda_ky,t_k+\lambda_k^2s),
\end{aligned}
$$
and  define 
$\tilde a_{2,k},\tilde b_{1,k},\tilde b_{2,k},\tilde c_{1,k},\tilde c_{2,k}$
analogously.
Then $(\tilde u_k,\tilde v_k)$ solve the system
$$
\begin{aligned}
\tilde u_t-d_1\Delta \tilde u &= \tilde u(\tilde a_{1,k}\lambda_k^2-\tilde b_{1,k}\tilde u+\tilde c_{1,k}\tilde v), \\
\tilde v_t-d_2\Delta \tilde v &= \tilde v(\tilde a_{2,k}\lambda_k^2-\tilde b_{2,k}\tilde v+\tilde c_{2,k}\tilde u),
\end{aligned}
$$
in the corresponding rescaled region, $\tilde u_k(0,0)+\tilde v_k(0,0)=1$
and $u_k+v_k\leq4$ in $\tilde D_k:=\{y\in\R^n:|y|<k/2\}\times(-k^2/4,k^2/4)$.
Passing to a subsequence we may assume 
$\tilde b_{i,k}(0,0)\to\tilde b_i>0$ and 
$\tilde c_{i,k}(0,0)\to\tilde c_i>0$, $i=1,2$, 
where $\tilde c_1\tilde c_2>\tilde b_1\tilde b_2$.
Now standard regularity estimates guarantee that a subsequence of 
$(\tilde u_k,\tilde v_k)$ converges to a nontrivial nonnegative solution
of \eqref{lv2}
with $X=\R^n$ and $b_1,b_2,c_1,c_2$ replaced by $\tilde b_1,\tilde b_2,\tilde c_1,\tilde c_2$,
which contradicts Theorem~\ref{thm-liouv1}.

If $\Omega$ is smooth
and the solution $(u,v)$ satisfies
the homogeneous Dirichlet or Neumann boundary condition
on $\partial\Omega\times(0,T)$
then one just has to modify the proof
in the same way as in the proof of \cite[Theorem 4.1]{PQS2}.
If $a_1=a_2=0$ and $b_1,b_2,c_1,c_2$ are constants
then it is sufficient to replace the inequality 
\eqref{Mkbig} with
$$M_k(y_k,\tau_k)>2k d_P^{-1}((y_k,\tau_k),\partial D_k)$$
and notice that $\lambda_k$ need not converge to zero 
(cf.~the proof of \cite[Theorem 3.1(ii)]{PQS2}).
\end{proof}

In the same way as in the proof Theorem~\ref{thm-ub},
using Theorem~\ref{thm-liouv-per} instead of Theorem~\ref{thm-liouv1}
and assuming \eqref{ass-per},
we obtain the following universal bounds for periodic solutions
of the homotopy problem
\begin{equation} \label{sys-hom}
\begin{aligned}
&\left.\begin{aligned}
u_t-\Delta u &= \lambda u(a_1-b_1u+c_1v)+(1-\lambda)(\Lambda u+K^3v^2) \\
v_t-\Delta v &= \lambda v(a_2-b_2v+c_2u)+(1-\lambda)(\Lambda v+u^2) 
\end{aligned}\quad\right\}\quad\hbox{in }\Omega\times(0,\infty), \\
&\qquad u=v=0\qquad\hbox{on }\partial\Omega\times(0,\infty),
\end{aligned}
\end{equation} 
where 
\begin{equation} \label{LK}
\lambda\in[0,1],\qquad \Lambda>0, \qquad K=K(x,t):=\frac{c_1+b_2}{c_2+b_1}.
\end{equation}

\begin{theorem} \label{thm-ub-per}
Assume \eqref{ass-per} and \eqref{LK}.
Then there exists a positive constant
$C$ dependening only on $\Omega,T,\Lambda,a_1,a_2,b_1,b_2,c_1,c_2$
such that any positive $T$-periodic solution $(u,v)$ of \eqref{sys-hom} satisfies
\begin{equation} \label{est5}
 u(x,t)+v(x,t)\leq C \quad\hbox{for all $(x,t)\in\Omega\times(0,\infty)$}.
\end{equation}
\end{theorem}

\section{Periodic solutions} \label{sec-ps}

In this section 
we assume \eqref{ass-per}.
In addition, by ${\cal X}:=BUC(\Omega\times(0,T))$
we denote the space of bounded uniformly continuous
functions equipped with the $L^\infty$-norm $\|\cdot\|_\infty$,
and $w^+(x,t):=\max(w(x,t),0)$. 
Without fearing confusion, by $\|\cdot\|_\infty$ we denote
both the norm in $L^\infty(\Omega\times(0,T))$ and $L^\infty(\Omega)$.

In the proof of our main result we will need the following
proposition on (possibly sign-changing) solutions.

\begin{proposition} \label{propBH}
Let $\Omega\subset\R^n$ be a $C^3$-smooth bounded domain, $T>0$ and
$f\in{\cal X}$.
Then the scalar periodic problem
\begin{equation} \label{eqn-per}
\left.\begin{aligned}
w_t-\Delta w &= f         &\qquad&\hbox{in }\Omega\times(0,T),\\
           w &= 0         &\qquad&\hbox{on }\partial\Omega\times(0,T),\\ 
  w(\cdot,0) &= w(\cdot,T)&\qquad&\hbox{in }\Omega,
\end{aligned}\quad\right\}
\end{equation}
possesses a unique solution $w$. In addition, 
the mapping ${\cal K}:{\cal X}\to{\cal X}:f\mapsto w^+$ is compact.
\end{proposition}

\begin{proof}
The assertion was proved in \cite{BH} for $f$ being H\"older continuous.
In our case it is suffcient to combine this result with standard
mollifying arguments and $L^\infty$- and smoothing estimates 
for the corresponding initial value problem.
For example, if $w$ is a periodic solution of \eqref{eqn-per}
then the variation-of-constants formula yields the estimate
$$ \|w(\cdot,t_2)\|_\infty\leq
e^{-\Lambda_1(t_2-t_1)}\|w(\cdot,t_1)\|_\infty+(t_2-t_1)\|f\|_\infty,$$
and, consequently, the periodicity of $w$ guarantees
$\|w\|_\infty\leq C(T)\|f\|_\infty$.
\end{proof}

We will also need the following proposition on the adjoint eigenvalue
problem.

\begin{proposition} \label{propD}
Let $\Omega\subset\R^n$ be a $C^3$-smooth bounded domain and $T>0$. 
Then there exists $\Lambda_1^T>0$ such that the problem
\begin{equation} \label{eqn-dual}
\left.\begin{aligned}
-\varphi_t-\Delta\varphi &= \Lambda_1^T\varphi  &\qquad&\hbox{in }\Omega\times(0,T),\\
           \varphi &= 0         &\qquad&\hbox{on }\partial\Omega\times(0,T),\\
  \varphi(\cdot,0) &= \varphi(\cdot,T)&\qquad&\hbox{in }\Omega,
\end{aligned}\quad\right\}
\end{equation}
possesses a positive solution $\varphi$.
\end{proposition}

\begin{proof}
The result follows again from \cite{BH}; cf.~also \cite[(18)]{Est}.
\end{proof}

\begin{proof}[Proof of Theorem~\ref{thm-per}]
Let ${\cal K}$ be the compact mapping from Proposition~\ref{propBH}.
First notice that fixed points of the compact operator
$$ \begin{aligned}
&{\cal T}:\,{\cal X}\times{\cal X}\to{\cal X}\times{\cal X} \\
&{\cal T}(u,v):=({\cal K}(u(a_1-b_1u+c_1v)),{\cal K}(v(a_2-b_2v+c_2u))),
\end{aligned}$$
correspond to nonnegative periodic solutions of our problem.
In fact, if $(u,v)\ne(0,0)$ is a fixed point of ${\cal T}$ then
$u=w^+$ and $v=z^+$, where $(w,z)$ are $T$-periodic solutions
of 
$$ 
\begin{aligned}
&\left.\begin{aligned}
w_t-\Delta w &=  w^+(a_1-b_1w^+ +c_1z^+) \\
z_t-\Delta z &=  z^+(a_2-b_2z^+ +c_2w^+) 
\end{aligned}\quad\right\}&\quad&\hbox{in }\Omega\times(0,\infty), \\
&\qquad\qquad w=z=0&\quad&\hbox{on }\partial\Omega\times(0,\infty),
\end{aligned}
$$
and the maximum principle guarantees $w,z\geq0$. 

Notice also that nontrivial nonnegative periodic solutions of \1D
are positive. In fact, by the maximum principle it is sufficient
to exclude the possibility $v\equiv0$ (or $u\equiv0$).
If, for example, $v\equiv0$ then the assumptions $a_1<\Lambda_1$
and $b_1>0$ guarantee that 
the function $t\mapsto\int_\Omega u(\cdot,t)\varphi_1\,dx$
(where $\varphi_1$ is a positive eigenfunction
of the negative Dirichlet Laplacian corresponding to
the eigenvalue $\Lambda_1$)
is time-decreasing which contradicts the periodicity of $u$.
In fact, multiplying the first equation in \eqref{lv11}
by $\varphi_1$ and integrating over $\Omega$ yields
$$
\frac{d}{dt}\int_\Omega u(\cdot,t)\varphi_1\,dx
 =\int_\Omega(a_1-\Lambda_1)u(\cdot,t)\varphi_1\,dx
 -\int_\Omega b_1 u^2(\cdot,t)\varphi_1\,dx<0. 
$$
We will prove
the existence of a nontrivial fixed point of ${\cal T}$
(hence a positive periodic solution of \1D)
by computing the Leray-Schauder degree
$d(r):=\hbox{deg}\,(I-{\cal T},B_r,0)$
for small and large $r$, where
$I$ denotes the identity and $B_r$ is the ball
in ${\cal X}\times{\cal X}$ with radius $r$  centered at zero.
In fact, we will prove $d(r)=1$ if $r>0$ is sufficiently small
and $d(r)=0$ if $r>0$ is large enough.

First consider $r$ small.
The assertion $d(r)=1$ follows by using the homotopy 
$${\cal T}_\lambda(u,v):=({\cal K}(\lambda u(a_1-b_1u+c_1v)),
                          {\cal K}(\lambda v(a_2-b_2v+c_2u))),
 \quad \lambda\in[0,1].$$
To show that this homotopy is admissible, assume 
that there exists a nontrivial fixed point $(u,v)$ of ${\cal T}_\lambda$
satisfying $\|(u,v)\|_\infty=r\ll1$.
Fix $t$ such that $\|(u(\cdot,t),v(\cdot,t))\|_\infty=r$.
W.l.o.g.\ we may assume $\|u(\cdot,t)\|_\infty=r$.
Notice that $(u,v)$ is a positive periodic solution of \1D with
the right-hand sides multiplied by $\lambda$,
and the variation-of-constants formula yields
$$ \begin{aligned}
r &= \|u(\cdot,t+T)\|_\infty\leq e^{-(\Lambda_1-\lambda\max a_1)T}\|u(\cdot,t)\|_\infty
   +CT\lambda(\|u\|_\infty^2+\|u\|_\infty\|v\|_\infty) \\
  &\leq e^{-(\Lambda_1-\max a_1)T}r+2CTr^2,
\end{aligned}
$$
which yields a contradiction for $r$ small.

Now consider $r$ large.
We will use the homotopy
$$ \begin{aligned}
{\cal T}_\lambda(u,v):=&
({\cal K}(\lambda u(a_1-b_1u+c_1v)+(1-\lambda)(\Lambda u+K^3v^2)),\\
&\ {\cal K}(\lambda v(a_2-b_2v+c_2u)+(1-\lambda)(\Lambda v+u^2))),
\quad \lambda\in[0,1],
\end{aligned}
$$
where $\Lambda:=\Lambda_1^T+1$.
Estimates in Theorem~\ref{thm-ub-per} guarantee that this homotopy
is admissible if $r$ is large enough.
Hence it is sufficient to show that problem
\eqref{sys-hom} does not possess positive periodic solutions
if $\lambda=0$.

Assume on the contrary that $(u,v)$ is a positive $T$-periodic
solution of the system
$$
\begin{aligned}
&\left.\begin{aligned}
u_t-\Delta u &= \Lambda u+K^3v^2 \\
v_t-\Delta v &= \Lambda v+u^2 
\end{aligned}\quad\right\}&\quad&\hbox{in }\Omega\times(0,\infty), \\
&\qquad\qquad u=v=0 &\qquad&\hbox{on }\partial\Omega\times(0,\infty).
\end{aligned}
$$
Multiplying the first equation by the eigenfunction $\varphi$ 
from Proposition~\ref{propD}, integrating over $\Omega\times(0,T)$
and using integration by parts we obtain
$$ \Lambda_1^T\int_0^T\int_\Omega u\varphi\,dx\,dt
  \geq \Lambda \int_0^T\int_\Omega u\varphi\,dx\,dt, $$
which yields a contradiction.
This concludes the proof. 
\end{proof}

\bibliographystyle{elsarticle-num}

\end{document}